\documentclass[11pt]{amsart}
\usepackage{amsmath}
\usepackage{amssymb}
\usepackage{amsthm}
\usepackage{mathrsfs}
\usepackage{comment}
\usepackage{hyperref}
\usepackage{algorithm,algpseudocode}
\algblock{Input}{EndInput}
\algnotext{EndInput}
\algblock{Output}{EndOutput}
\algnotext{EndOutput}

\usepackage[all,cmtip]{xy}\usepackage{xcolor}
\usepackage{enumerate}
\usepackage{bm}
\usepackage{dsfont}
\usepackage{mathtools}
\usepackage[cal=euler]{mathalfa}
\usepackage[top=1in, bottom=1.25in, left=1.25in, right=1.25in]{geometry}
\usepackage{parskip}

\hypersetup{colorlinks=true,linkcolor=magenta,citecolor=blue}

\DeclareMathAlphabet{\mathpzc}{OT1}{pzc}{m}{it}

\theoremstyle{plain}
\newtheorem{theorem}{Theorem}[section]
\newtheorem*{theorem*}{Theorem}

\newtheorem{lemma}[theorem]{Lemma}
\newtheorem*{claim*}{Claim}
\newtheorem{proposition}[theorem]{Proposition}

\newtheorem{corollary}[theorem]{Corollary}

\theoremstyle{definition}
\newtheorem{definition}[theorem]{Definition}
\newtheorem{defn}[theorem]{Definition}

\newtheorem{example}[theorem]{Example}

\newtheorem{remark}[theorem]{Remark}

\numberwithin{equation}{section}
\numberwithin{figure}{section}

\newcommand{\ZZ}{\mathbb{Z}}
\newcommand{\bbZ}{\mathbb{Z}}

\newcommand{\val}{v}
\newcommand{\RR}{\mathcal{R}} 
\newcommand{\PP}{\mathcal{P}} 

\newcommand{\intR}{\mathfrak{o}}

\newcommand{\normEF}{N_{E/F}}

\newcommand{\normKE}{N_{K/E}}
\newcommand{\normKF}{N_{K/F}}
\newcommand{\normKFim}{N_{K/F}(K^\times)}

\newcommand{\Kone}[1]{K[t;\sigma]/K[t;\sigma](t^{\de}-#1)}

\DeclareMathOperator{\id}{id}
\newcommand{\os}{n}  
\newcommand{\de}{m}  
\newcommand{\Q}{\mathbb{Q}}
\newcommand{\N}{\mathbb{N}}
\newcommand{\F}{\mathbb{F}}
\newcommand{\Z}{\mathbb{Z}} 

\DeclareMathOperator{\Gal}{Gal}
\DeclareMathOperator{\Aut}{Aut}

\newcommand{\prootF}{\xi_0}
\newcommand{\prootK}{\xi}
\newcommand{\itemref}[1]{\textbf{\textcolor{magenta}{P}\ref{List:#1}}}
\newcommand{\gcdesn}{[es]_{\frac{q-1}{n}}}
\newcommand{\Thetaes}{[es]_{q-1}}

\begin{document}
\title{A parametrization of nonassociative Petit algebras $K[t;\sigma]/K[t:\sigma](t^m-a)$ up to equivalence}

\author{Monica Nevins}
\address{Department of Mathematics and Statistics, University of Ottawa, Ottawa, Canada K1N 6N5}
\email{mnevins@uottawa.ca}
\thanks{The first author's research is supported by NSERC Discovery Grant RGPIN-2025-05630.}

\author{Susanne Pumpl\"un}
\address{School of Mathematical Sciences, University of Nottingham,
Nottingham NG7 2RD}
\email{Susanne.Pumpluen@nottingham.ac.uk}
\thanks{Corresponding author: Susanne Pumpl\"un. Email: \href{mailto:Susanne.Pumpluen@nottingham.ac.uk}{Susanne.Pumpluen@nottingham.ac.uk}}

\keywords{Petit algebras, nonassociative algebras, skew constacyclic codes, LWE}

\subjclass[2020]{Primary 17A35; Secondary 17A60, 17A99, 11T71}

\date{\today}

\begin{abstract}
Let $K/F$ be a cyclic Galois field extension of degree $n$ with
$\mathrm{Gal}(K/F)=\langle\sigma\rangle$. The nonassociative Petit algebras
$K[t;\sigma]/K[t;\sigma](t^m-a)$ have recently emerged as important tools in coding theory and cryptography, where the applications only require distinct algebras up to isomorphism. Algebras that are isomorphic via a monomial isomorphism of degree one are called equivalent.  We present a novel and precise parametrization of the $F$-equivalence classes of these algebras in the key case that $F$ contains a primitive $n$th root of unity and $n\mid m$. We evaluate the parametrization explicitly when $F$ is a nonarchimedean local field, which has applications to the number field case recently used to construct a nonassociative variant of Learning with Errors.
\end{abstract}

\maketitle

\section{Introduction}

Let  $K/F$ be a cyclic Galois field extension of degree $n$ and let $\sigma$ generate
$\mathrm{Gal}(K/F)$.
Petit algebras are defined using skew polynomials and were discovered by Petit  \cite{Petit1967}: For a monic skew polynomial $f\in K[t;\sigma]$ of degree $m>1$, a Petit algebra
$\mathbb S_f=K[t;\sigma]/K[t;\sigma]f$ is a unital nonassociative $F$-algebra  of dimension $mn$ whose multiplication is defined via
right division by $f$.

Petit algebras appear in two applications.  Petit algebras with zero divisors serve as
ambient rings for skew polycyclic codes, and their
isomorphisms are used to classify such codes up to equivalence and isometry
\cite{Pumpluen2025, NevinsPumpluen2025submitted}.
Independently, Mendelsohn and Ling \cite{MendelsohnLing2025} used orders in
  $\mathbb{S}_{t^m-a}$ to construct a structured
variant of Learning with Errors (LWE), a key problem in modern cryptography,   when $m=n$, $a\in K\smallsetminus F$  and $K$ is a number field.
A full
search-to-decision reduction without restriction to the secret space was achieved. In contrast, using orders in associative cyclic division algebras, does  require restriction on the secret space (see cyclic LWE in \cite{Groveretall2022}).  It is therefore natural to ask whether more general proper nonassociative algebras $\mathbb S_{t^m-a}$   can be used in LWE as well, and a parametrization up to certain isomorphisms tells us which choices of $\sigma$ and $a\in K^\times$ yield genuinely different algebras.

In this paper we focus on the class of proper nonassociative $F$-algebras $\mathbb{S}_{t^m-a}$ where $n=[K:F]$ divides $m$ and $a\notin F$.  Two such algebras with distinct $K$ are never isomorphic (Lemma~\ref{Lem:differentextensions}) and if $\Gal(K/F)$ is abelian, then the results of \cite[Theorem 3.3]{Pumpluen2025closerlook} show that distinct choices of generators $\sigma$ of $\Gal(K/F)$ yield non-isomorphic algebras as well.   Thus it remains to parametrize the proper nonassociative algebras up to isomorphism, for fixed choices of $K$ and $\sigma$. Those isomorphisms that are given by monomial isomorphisms fixing the base field $F$ are called $F$-isometries. In prior work \cite{NevinsPumpluen2025submitted} we have established when all $F$-isometries are of degree one (which we call $F$-equivalences), that is, by isomorphisms $G_{\tau,\alpha}$ for some $\tau\in \Gal(K/F)$ and $\alpha\in K^\times$ such that $G(ct)=\tau(c)\alpha t$.  By the results of \cite{NevinsPumpluen2025submitted}, when $m=ns$ then two algebras $\mathbb{S}_{t^m-b}$ and $\mathbb{S}_{t^m-a}$ are $F$-equivalent
if and only if there exist $\alpha\in K^\times$ and $\tau\in \Gal(K/F)$ such that $$\tau(a) =  N_{K/F}(\alpha)^sb.$$ In Theorem~\ref{T:parametrizingnm}, we prove that this implies that  the $F$-equivalence
 classes correspond to $\Gal(K/F)$ orbits on $K^\times/\normKFim^s$. Our aim is to make this description explicit.

Our  main result gives such a parametrization when $F$ contains a primitive $n$th
root of unity, so that $K/F$ is a Kummer extension. Writing elements of $K$ in a Kummer
basis $1,\beta,\dots,\beta^{n-1}$ with $\sigma(\beta)=\zeta\beta$, we construct an
explicit subset $\mathcal{S}_s(K)\subset K^\times$ (Definition~\ref{D:SsK}) which achieves this goal.

\begin{theorem*} (Theorem \ref{Main:general})
Suppose $F$ contains a  primitive $n$th root of unity $\zeta$ and that $m=ns$ for some integer $s$.  Let $t_s>0$ be the least positive integer such that $\zeta^{t_s}\in N_{K/F}(K^\times)^s$.
 Then for fixed $K$ and $\sigma$, the distinct
 $F$-equivalence classes of the  Petit algebras $\mathbb{S}_{t^m-a}$
with  $a\in K^\times$ are represented by $a \in \mathcal{S}_s(K)$.
\end{theorem*}

For $m=n$ this yields a complete set of representatives of the $F$-isomorphism classes
of proper nonassociative cyclic algebras of arbitrary degree $n$
(Corollary~\ref{C:nonassociativecyclic}). This was previously known only for prime $n$
\cite{NevinsPumpluen2025}; note that we also correct a small omission in that result
(Remark~\ref{re:correction}).

Our parametrization also has a consequence for general Petit algebras. Let
$f(t)=t^m-\sum_{j=0}^{m-1}a_jt^j$ and $g(t)=t^m-\sum_{j=0}^{m-1}b_jt^j$ in
$K[t;\sigma]$, and suppose there is an $F$-equivalence
$G_{\tau,\alpha}: \mathbb S_f\to \mathbb S_g$. Then for each $j$ with $n\mid(m-j)$ and $a_j\neq 0$,
the coefficients $a_j$ and $b_j$ have the same support with respect to the Kummer basis,
that is, they lie in the same part $K(I_j)$, $I_j\in\mathcal{P}$, of the partition of $K$
from Definition~\ref{def:K(I)}.  (Corollary \ref{cor:generalcase}).

  To put our results into context, in \cite{Pumpluen2025closerlook}, the second author  counted the isomorphism classes of the Petit semifields $\mathbb{S}_{t^m-a}$  with respect to \emph{all} ring isomorphisms $G_{\tau,\alpha}$, $\tau\in\mathrm{Aut}(K)$  \cite[Theorem 5.7]{Pumpluen2025closerlook}, and  determined the number of isotopy classes of all
Petit semifields $\mathbb{S}_f$ as well \cite[Theorem 6.1, Theorem 6.2]{Pumpluen2025closerlook}.
In contrast, the present paper focuses on $F$-isomorphisms which are $F$-equivalences and gives explicit methods to construct a system of representatives.

In Section~\ref{sec:local} we evaluate $\mathcal{S}_s(K)$ explicitly when $F$ is a
nonarchimedean local field with residue field of order $q$, and $K/F$ is cyclic of
degree $n=ef$ with $n\mid(q-1)$ and $\gcd(m,p)=1$. All ingredients of the
parametrization are then given by closed formulas in $e$, $f$, $s$ and $q$, and the
representatives of the remaining finite quotients are produced by  Algorithm~\ref{Algorithm:1}.

The paper is organized as follows. Section~\ref{sec:prelim} collects the necessary
background on Petit algebras and their isomorphisms. Section~\ref{sec:main} contains the
parametrization for $n\mid m$,  and Section~\ref{sec:local} its evaluation over local nonarchimedean fields.
Our  results immediately give  necessary conditions on $F$-equivalences over number fields $F$ via completion (Section~\ref{sec:number fields}).

\section{A collection of terminology and previous results}\label{sec:prelim}

\subsection{Twisted polynomial rings}
Let  $K/F$ be a cyclic Galois field extension of degree $n$ with Galois group $G$ generated by $\sigma$.

 The \emph{twisted polynomial ring} $R=K[t;\sigma]$
is the ring of polynomials $\{a_0+a_1t+\dots +a_mt^m\,|\, a_i\in K\}$, with term-wise addition and multiplication given by the rule
$ta=\sigma(a)t$ for all $a\in K$ \cite{Ore1933}; see also \cite[Chapter I]{Jacobson1996}.  The units of $R$ are  the constant nonzero polynomials $K^\times$.

For $f=a_0+a_1t+\dots +a_mt^m\in R$ with $a_m\not=0$ define ${\rm
deg}(f)=m$ and put ${\rm deg}(0)=-\infty$. Then ${\rm deg}(fg)={\rm deg}
(f)+{\rm deg}(g).$
 An element $f\in R$ is called \emph{irreducible} if  it is not a unit and it has no proper factors, \emph{i.e.} there do not exist
 $g,h\in R$, neither a unit, such that $f=gh$.
 There exists a right division algorithm in $R=K[t;\sigma]$: for all $g,f\in R$, $f\neq 0$, there exist unique $r,q\in R$
  such that ${\rm deg}(r)<{\rm deg}(f)$ and $g=qf+r$ \cite[\S1.1]{Jacobson1996}.

\subsection{Petit algebras, nonassociative cyclic algebras and Sandler's semifields} \label{subsec:nonass}

Let $f \in K[t;\sigma]$ have degree $m>1$ and let ${\rm mod}_r f$ denote the remainder of right division by $f$.
 The vector space $K[t;\sigma]_m=\{g\in K[t;\sigma]\,|\, {\rm deg}(g)<m\}$ together with the multiplication
 $g\circ h=gh \,\,{\rm mod}_r f $, where the right hand side is the remainder obtained after dividing $gh$ on the right by $f$,
 is a unital nonassociative algebra  over $F$ of dimension $mn$
  denoted by $K[t;\sigma]/K[t;\sigma]f$ or $\mathbb{S}_f$, and called a \emph{Petit algebra}, and a \emph{proper Petit algebra} when $\mathbb{S}_f$ is not associative.

  When ${\rm deg}(g)+{\rm deg}(h)<m$, the multiplication of $g$ and $h$ in $\mathbb{S}_f$ is thus the same as the multiplication of $g$ and $h$ in $R$. The algebra $\mathbb{S}_f$ is associative if and only if $f$ is \emph{two-sided},  (also called \emph{right-invariant}), meaning $(f)=Rf$ is a two-sided ideal.
In that case, $\mathbb{S}_f$ is the usual quotient algebra  $R/(f)$ \cite[(7),(9), (10)]{Petit1967}.

 If  $\mathbb{S}_f$ is a proper Petit algebra, then $\mathbb{S}_f$ has center $F$,
${\rm Nuc}_l(\mathbb{S}_f)={\rm Nuc}_m(\mathbb{S}_f)=K$
and
$${\rm Nuc}_r(\mathbb{S}_f)=\{g\in \mathbb{S}_f\,|\, fg\in Rf\}$$
 is the eigenspace of $f$ \cite{Petit1967}.

Let $a\in K^\times$.  The algebra $\mathbb{S}_{t^m-a}$
 is  associative if and only if $n\mid m$ and $a\in F^\times$ \cite{Petit1967}.  Thus we know that  $f(t)=t^m-a\in K[t;\sigma]$ is not two-sided, if and only if either $n\nmid m$ or  $ n\mid m$ but $a\in K\smallsetminus F$.

We remark that over finite fields, the proper nonassociative Petit division algebras are isotopic to the Jha-Johnson semifields (also called cyclic semifields) from \cite{JhaJohnson1989} and make up a huge class of semifields.
Indeed, for $n\geq m$, the proper nonassociative Petit division algebras $\F_{q^n}[t;\sigma]/\F_{q^n}[t;\sigma](t^m-a)$ are the opposite algebras of the well-known Sandler semifields \cite{Sandler1962}.

 In this paper we will investigate proper nonassociative Petit  algebras, where $f(t)=t^m-a\in K[t;\sigma]$ is not two-sided. The algebra $\mathbb{S}_{t^m-a}$ is a division algebra over $F$ if and only if $f(t)=t^m-a\in K[t;\sigma]$ is irreducible \cite[(7)]{Petit1967}. We also know that $\mathbb{S}_f$ is a division algebra if and only if  ${\rm Nuc}_r(\mathbb{S}_f)$ is a division algebra, e.g. see \cite[Proposition 4]{Gomez2014}.
For finite fields $K$ and for $n\geq m$, the opposite algebras of $\mathbb{S}_{t^m-a}$ are called \emph{Sandler's semifields} (if they are division algebras).

When $f(t)=t^m-a$ is irreducible has to be decided on a one-by-one basis.
When $m$ is prime we know more.

\begin{theorem} \cite[Theorem 3.11]{BrownPumpluen2018}\label{thm:prime}
Suppose that $m=2$ or $3$, or that $m\geq 5$ is a prime and that $F$ contains a primitive $m$th root of unity.
Then  $\mathbb{S}_{t^m-a}$ is a division algebra if and only if
$$\sigma^{m-1}(z)\cdots \sigma(z) z\neq a$$
for all $z\in K$.
\end{theorem}

\begin{lemma}\label{Lem:differentextensions} Let $K$ and $K'$ be two distinct cyclic Galois extensions of  $F$ of degree $n$.  Suppose that either $n\nmid m$ or that $ n\mid m$ but $a\in K\smallsetminus F$. Then for any two generators $\sigma,\sigma'$  of their respective Galois groups, and all $a'\in K'$
$$K[t;\sigma]/K[t;\sigma](t^m-a)\not\cong K'[t;\sigma']/K'[t;\sigma'](t^m-a').$$
\end{lemma}

This is an immediate consequence of the fact that the nuclei of two isomorphic proper nonassociative algebras are the same.

In many cases the different generators of the Galois group yield different families of mutually nonisomorphic proper nonassociative algebras when $a_i\in K\smallsetminus F$.

\begin{theorem}\label{thm:2.3}
Suppose that $\sigma_1, \sigma_2$ are two distinct generators of $\Gal(K/F)$, and let $F_0\subset F$ be any subfield such that $\Aut_{F_0}(K)$ is abelian.
\begin{enumerate}[(i)]
\item \label{thm:main1}  \cite[Theorem 3.1]{Pumpluen2025closerlook} Suppose $n \nmid m$. Then for every $a_1,a_2\in K^\times$, there is no isomorphism of $F_0$-algebras
$$
K[t_1;\sigma_1]/K[t_1;\sigma_1](t_1^m-a_1) \rightarrow K[t_2;\sigma_2]/K[t_2;\sigma_2](t_2^m-a_2)
$$
\item \label{thm:main2}  \cite[Theorem 3.2]{Pumpluen2025closerlook}  Suppose $n=\de$.  Then for every $a_1 \in K^\times$ and $a_2\in K\smallsetminus F$, there is no isomorphism of $F_0$-algebras
    $$
K[t_1;\sigma_1]/K[t_1;\sigma_1](t_1^\de-a_1) \to K[t_2;\sigma_2]/K[t_2;\sigma_2](t_2^\de-a_2).
    $$
\end{enumerate}
    That is in both (i) and (ii), varying $\sigma$ over the distinct generators of $\Gal(K/F)$ gives a partition of the $F_0$-isomorphism classes of such algebras, for every such $F_0$.
\end{theorem}

\begin{corollary}\label{cor:Euler} \cite[Corollary 3.3]{Pumpluen2025closerlook}
Let either $n \nmid m$ or, alternatively, let  $n =m $ and suppose that $a\in K\smallsetminus  F$.  Let $\varphi$ be the Euler function. Then there are $\varphi(n)$ different classes of proper Petit algebras of type $K[t;\sigma_i]/K[t;\sigma_i](t^m-a)$; one class for each choice of generator $\sigma_i$ of $\Gal(K/F)$. Algebras in two different classes are not isomorphic as $F$-algebras.
\end{corollary}

\subsection{Isomorphisms, equivalences and isometries}\label{sec:isometries}

Define $N_i^\sigma(a)=\sigma^{i-1}(a)\cdots \sigma(a)a$ for all positive integers $i$ and $N_0^\sigma(a)=1$. Note that $\beta \tau(N_i^\tau(\beta))=N_{i+1}^\tau(\beta)$ and thus
\begin{equation}\label{E:Normrelation}
N_{i+j}^\tau(\beta)= N_{i}^\tau(\beta) \cdot \tau^{i}(N_j^\tau(\beta))
\end{equation}
for all $i,j\geq 0$ and automorphisms $\tau$ of $K$.  When $\sigma$ has order $\os$, then  $N_{\os}^\sigma$ is simply the norm  $N_{K/F}$ of the field extension $K/F$.

\begin{definition}\label{Definition:isometryeuivalence}
The two nonassociative rings $\mathbb{S}_f=K[t;\sigma]/K[t;\sigma]f$ and $\mathbb{S}_h=K[t;\sigma]/K[t;\sigma]h$ are called \emph{isometric}, if there exists an isomorphism $$G=G_{\tau,\alpha,k}:\mathbb{S}_f\to \mathbb{S}_h$$
such that $G|_K=\tau\in {\rm Aut}(K)$ and $G(t)=\alpha t^k$ for some integer $k\geq 1$, and some $\alpha\in K^\times$, and \emph{equivalent} if $k=1$.
The isomorphism
$G_{\tau,\alpha,k}$ is  called an \emph{isometry} or a \emph{monomial isomorphism of degree $k$},
     $G_{\tau,\alpha,1} $ is called an \emph{equivalence}
 When $k=1$ we use the notation  $G_{\tau,\alpha}$.
 We call $G_{\tau,\alpha,k}$ an \emph{$F$-isometry} when $G_{\tau,\alpha,k}$ is $F$-linear, respectively  $G_{\tau,\alpha}$ an \emph{$F$-equivalence}, when $G_{\tau,\alpha}$ is $F$-linear  (i.e. $\tau\in \Gal(K/F)$).
\end{definition}

We have
$$G_{\tau,\alpha}(\sum_{i=0}^{\de -1}d_it^i)=\sum_{i=0}^{\de -1}\tau(d_i) N^\sigma_{i}(\alpha) t^i.$$

 In a large number of cases every isomorphism between proper nonassociative rings of the form $\mathbb{S}_{t^\de-a}$
is  monomial, and even stronger, has degree one:

\begin{theorem}\label{T:classifyisomorphism}  \cite[Section 2.1]{BrownPumpluen2018},\cite[Corollary 6.2, Corollary 6.4]{NevinsPumpluen2025submitted}
Assume that one of the following holds:
\begin{itemize}
\item $n\nmid m$,
  \item $n\mid m$, $a\in K\smallsetminus F$ and the technical condition
    \begin{equation}\label{stareq}
        \tag{$\star$} \text{if $s>1$ is the least value for which $sk\geq m$, then $(s+1)k\leq 2 m$},
    \end{equation}
     in \cite{NevinsPumpluen2025submitted} holds for every isomorphism $G$ between two rings $K[t;\sigma]/K[t;\sigma](t^m-a)$ and $K[t;\sigma]/K[t;\sigma](t^m-b)$:
\end{itemize}
Then
$$
K[t;\sigma]/K[t;\sigma](t^m-a)\cong K[t;\sigma]/K[t;\sigma](t^m-b)
$$
as $F$-algebras
 if and only if there exists some $\tau\in \Gal(K/F)$ and some $\alpha\in K^\times$  such that
\begin{equation}\label{equivrelation}
 \tau(a)=N_m^\sigma (\alpha)b.
\end{equation}
The isomorphism is given by $G_{\tau,\alpha}$.
\end{theorem}

 Thus there may be some non-monomial isomorphisms between proper nonassociative Petit rings which can only occur when their degree does not satisfy \eqref{stareq},  $n\mid m$ and $a$ or $b$ lie outside of $F$.

The equivalences between $K[t;\sigma]/K[t;\sigma](t^m-a)$ and $K[t;\sigma]/K[t;\sigma](t^m-b)$ play an important role when classifying the general Petit rings $\mathbb{S}_f$ over $F$, too.

\begin{corollary}  \label{cor:equivpolycyclic1}
Let $f(t)=t^m-\sum_{i=0}^{m-1}a_it^i, h(t)=t^m-\sum_{i=0}^{m-1}b_it^i\in K[t;\sigma]$.
If
$G_{\tau,\alpha}:K[t;\sigma]/K[t;\sigma]f\rightarrow K[t;\sigma]/K[t;\sigma]h$ is an isomorphism, then for $\beta_i=\sigma^i(\alpha)$,
  $$G_{\tau,\beta_i}:K[t;\sigma]/K[t;
  \sigma](t^{m-i}-a_i)\to K[t;\sigma]/K[t;\sigma](t^{m-i} -b_i)$$
is an isomorphism for all $i \in\{ 0, \ldots, m-1\}$ such that  $(a_i,b_i)\not=(0,0)$.
\end{corollary}

\begin{proof}
The algebras  $\mathbb{S}_f$ and $\mathbb{S}_h$ are equivalent   via $G_{\tau,\alpha}$, if and only if there exists $\alpha \in K^{\times}$ and $\tau\in {\rm Aut}(K)$, such that
$$\tau(a_i) = N_{m-i}^\sigma ( \sigma^i(\alpha) ) b_i$$
for all $i \in\{ 0, \ldots, m-1\}$. Put $\beta_i=\sigma^i(\alpha)$,
 then
$$\tau(a_i) = N_{m-i}^\sigma ( \beta_i ) b_i$$
which implies that the map 
$G_{\tau,\beta_i}: \mathbb{S}_{t^{m-i}-a_i}\to \mathbb{S}_{t^{m-i}-b_i}$ is an isomorphism for all $i \in\{ 0, \ldots, m-1\}$ such that  $(a_i,b_i)\not=(0,0)$.
\end{proof}

\subsection{Counting $F$-equivalence classes using $G_{\tau,\alpha}$ when $n \mid m$}

We first fix a choice of cyclic field extension $K/F$ of degree $n$ (Lemma \ref{Lem:differentextensions}).

 We want to classify the proper nonassociative rings
  $\mathbb{S}_{t^m-a}$
 up to $F$-equivalence and call the resulting classes \emph{$F$-equivalence classes}.
We do not assume that $t^m-a\in K[t;\sigma]$ is irreducible.

\begin{theorem}\label{T:parametrizingnm}
    Suppose $K$ is a cyclic Galois extension of degree $n$.  If $m=ns$, with $s\in \mathbb{N}$, then the  $F$-equivalence classes  of Petit algebras of the form $\Kone{a}$, for some $a\in K^\times$,  are parametrized by the orbits under $\Gal(K/F)$ of the quotient group $K^\times/N_{K/F}(K^\times)^s$.
\end{theorem}

 By Theorem~\ref{T:classifyisomorphism},  if we additionally have  $a\in K\smallsetminus F$ and the technical condition (*) in \cite{NevinsPumpluen2025submitted} holds  for every isomorphism $G$ between two rings $K[t;\sigma]/K[t;\sigma](t^m-a)$ and $K[t;\sigma]/K[t;\sigma](t^m-b)$, too, then these are indeed all the possible $F$-isomorphisms between the two algebras.

\begin{proof}
By Theorem~\ref{T:classifyisomorphism}, we have that
\begin{equation}\label{E:iso}
K[t;\sigma]/K[t;\sigma](t^m-a)\cong K[t;\sigma]/K[t;\sigma](t^m-b)
\end{equation}
via an $F$-equivalence $G_{\tau,\alpha}$ if and only if there exists some $\tau\in \Gal(K/F)$ and some  $\alpha\in K^\times$  such that
$$ a=\tau(b) N_m^\sigma(\alpha).$$
Since $m=ns$,  this condition simplifies to
$$ a = \tau(b)N_{K/F}(\alpha)^s.$$
Since $N_{K/F}(\alpha)^s\in F^\times$, it is fixed under $\tau$.  Letting now $\tau$ vary over $\Gal(K/F)=\langle \sigma \rangle$ yields that for fixed $b\in K$, \eqref{E:iso} holds for every
$$
a \in \bigcup_{0\leq k \leq n-1}\sigma^k(b)N_{K/F}(K^\times)^s.
$$
That is, $a$ must lie in one of these cosets of the subgroup $N_{K/F}(K^\times)^s$ of $K^\times$, and these cosets form a single $\Gal(K/F)$-orbit in $K^\times/N_{K/F}(K^\times)^s$.
\end{proof}

To simplify notation, when $m=ns$ we write $a\simeq b$ if and only if there exist
$\tau\in\mathrm{Gal}(K/F)$ and $\alpha\in K^\times$ such that
$G_{\tau,\alpha}\colon K[t;\sigma]/K[t;\sigma](t^m-a)\to K[t;\sigma]/K[t;\sigma](t^m-b)$
is an isomorphism. Equivalently, $a\simeq b$ if and only if
$a\in\sigma^i(b)N_{K/F}(K^\times)^s$ for some $i\in\mathbb{Z}$.

\section{An explicit parametrization  when $n\mid m$}\label{sec:main}

In this section we give a parametrization of  the Petit rings $K[t;\sigma]/K[t;\sigma](t^m-a)$ as $F$-algebras up to $F$-equivalence, that is up to $F$-isomorphisms $G_{\tau,\alpha}$, assuming that $n\mid m$. These  techniques work for arbitrary fields $F$  that contain a primitive $n$th root of unity, where the extension $K/F$ can be described by Kummer theory.

We begin in Section~\ref{SS:Kummer} with a brief introduction to Kummer theory and a first partition of $K$ into Galois-stable subsets.  We then assemble the ingredients of the parametrization in Section~\ref{SS:parametrization}, and state the main theorem, Theorem~\ref{Main:general}.  We prove it in Section~\ref{SS:proof} and derive several consequences.

\subsection{Kummer extensions}\label{SS:Kummer}
 Let $F$ be a field and $n>1$ be an integer.  Assume that $F$ contains a primitive $n$th root of unity $\zeta$.  If $F$ has prime characteristic $p$, then since the $p$-power map is an automorphism we infer that $p \nmid n$.

By Kummer theory, the  Galois extensions of $F$ of degree $n$ are in bijection with the subgroups of  $F^\times/(F^\times)^n$ of order $n$ \cite[Ch. VI, Theorem 6.2, 8.2]{LangAlgebra}.  Furthermore, each  cyclic Galois field extension of degree $n$ is obtained as $K=F[x]/(x^n-b)$ for some $b\in F^\times$ such that $b$ is not an $r$th power, for any divisor $r>1$ of $n$.

Let $K$ be such a cyclic extension and let $\Gal(K/F)=\langle \sigma \rangle$. We may choose $\beta\in K^\times$ to be a root of $x^n-b$ for which $\sigma(\beta) = \zeta \beta$.
Then
\begin{equation}\label{E:Kummerbasis}
B=\{1, \beta, \beta^2, \ldots,\beta^{n-1}\}
\end{equation}
is an $F$-basis for $K$ with the property that $\sigma(\beta^k)=\zeta^k\beta^k$ for all $0\leq k <n$ \cite[Ch. VI, Theorem 6.2]{LangAlgebra}.

We now partition $K$ into subsets that are invariant both under $\Gal(K/F)$ and left multiplication by elements of $F^\times$.  Due to the particularly elegant nature of Kummer extensions, this can be done by specifying the nonzero coefficients of an element $a\in K$ with respect to the basis $B$ of \eqref{E:Kummerbasis}.

\begin{defn}\label{def:K(I)}
For every subset $I \subset \{0, 1, \ldots, n-1\}$, define
$$
K(I) = \{ \sum_{i=0}^{n-1} a_i \beta^i \mid \forall i\in I, a_i\in F^\times; \forall i\notin I, a_i=0\} \subset K.
$$
\end{defn}

Let $\mathcal{P}$ denote the set of all subsets of $\{0,1,\ldots, n-1\}$ and let $\mathcal{P}_0=\{I \in \mathcal{P} \mid I\neq \emptyset, I\neq \{0\}\}$.

\begin{lemma}\label{lemma:KminusFisP0}
    For every $I\in \mathcal{P}$, the set $K(I)$ is invariant under both $\Gal(K/F)$ and multiplication by elements of $F^\times$.  Moreover,  $K$ is the disjoint union over all $I\in \mathcal{P}$ of the sets $K(I)$, and
    $$
K \smallsetminus F = \bigcup_{I\in \mathcal{P}_0} K(I).
$$
\end{lemma}

\begin{proof}
Let $I\in \mathcal{P}$ and suppose $a=\sum_{i\in I}a_i\beta^i\in K(I)$. Then
$$
\sigma(a)=\sum_{i\in I}\sigma(a_i\beta^i) = \sum_{i\in I} a_i\zeta^i\beta^i
$$
since $\sigma(a_i)=a_i$ for all $i$ and $\sigma(\beta^i)=\zeta^i\beta^i$.  As $a_i\zeta^i\in F^\times$ for all $i\in I$, we conclude that $\sigma(a)\in K(I)$, and thus $K(I)$ is invariant under $\Gal(K/F)=\langle \sigma \rangle$.  Similarly, left multiplication by elements of $F^\times$ does not change the positions of the nonzero coefficients of $a$ relative to $B$, and so preserves the set $K(I)$.

    Every $a\in K$ can be written in the form $\sum_ia_i\beta^i$; choosing $I=\{i\mid a_i\neq 0\}$ yields $a\in K(I)$.  By construction, these sets are disjoint, so the first statement holds.  The element $a$ lies in $F$ if and only if $a_i=0$ for all $i>0$, that is, if and only if either $a\in K(\{0\})$ (so that $a\in F^\times$) or $a\in K(\emptyset)$ (so that $a=0$).
\end{proof}

\subsection{A parametrization of the $F$-equivalence  classes  when $n\mid m$} \label{SS:parametrization}

From now on,  assume that $K/F$ is a Kummer extension as in Section~\ref{SS:Kummer} and that $n\mid m$.
Our goal in this section is to give a set of representatives $a$ for the distinct
degree one $F$-isometry classes of algebras $\Kone{a}$, generalizing ideas presented in \cite{NevinsPumpluen2025}.

We begin by applying Theorem~\ref{T:parametrizingnm} to  the partition of $K =\bigcup_{I\in \mathcal{P}}K(I)$.

We first by note that the distinct subsets $I\in \mathcal{P}$ partition the equivalence classes.

\begin{lemma}\label{IJdistinct}
Suppose $a\in K(I)$ and $b\in K(J)$, for some $I, J\in \mathcal{P}$. If $a\simeq b$ then we must have $I=J$.
\end{lemma}

\begin{proof}[Proof of Lemma~\ref{IJdistinct}]
    Since $K(J)$ is invariant under $\sigma$ and under multiplication by $F^\times$, it is in particular invariant under $\sigma^i$ for each $i$ and also under multiplication by $N_{K/F}(K^\times)^s\subset F^\times$.  Thus $a\simeq b$ implies $I=J$.
\end{proof}

It thus suffices to decompose each $K(I)$ into $\Gal(K/F)$-orbits of cosets of $\normKFim^s$.

Let us set our notation.  We have $s=m/n$.
For each nonempty $I\in \mathcal{P}$, we index the elements of $I$ as
$$
i_0<i_1<\cdots<i_k.
$$
If $k\geq 1$ then write $F^{[\times k]}$ for the $k$-fold direct product $F^\times\times F^\times \times \cdots \times F^\times$, not to be confused with the set of powers $(F^\times)^s = \{x^s \mid x\in F^\times\}$.

In the following paragraphs, we will define the following elements:
\begin{enumerate}[\textbf{P}1.] \setlength{\itemindent}{10pt}
    \item \label{List:1} A primitive $n$th root of unity $\zeta$, and the size $n/t_s$ of the intersection $\mu_n\cap \normKFim^s$;
    \item \label{List:2} A set of representatives $\mathcal{D}_s$ of the quotient $F^\times/\mu_n\normKFim^s$ (Lemma~\ref{L:FmunNormsrepresentatives});

    \item \label{List:3} For each $I\in \mathcal{P}$, with $|I|=k+1\geq 2$, a cyclic group $\Delta_{t_s,I}$ (Definition~\ref{D:DeltatsI}) of $F^{[\times k]}$;
    \item \label{List:4} A set of representatives $\mathcal{E}_{s,I}$ for the quotient $F^{[\times k]}/\Delta_{t_s,I}$;
    \item \label{List:5} Subsets $K(t_s,I;a_{i_0})$ of $K(I)$ (Definition~\ref{D:KtIa}); and, finally
    \item \label{List:6} The set $\mathcal{S}_s(K)$ (Definition~\ref{D:SsK}) that we prove, in  Theorem~\ref{Main:general}, parametrizes these algebras.
\end{enumerate}

We begin with \itemref{1}. Write
$$\mu_n = \langle \zeta \rangle \subset F^\times$$ for the group of $n$th roots of unity.  Since $\mu_n \cap N_{K/F}(K^\times)^s$ will be a subgroup of $\mu_n$, its order is $n/t_s$  for some $t_s$ dividing $n$.  Then in particular we have
\begin{equation}\label{E:t_0}
\mu_n\cap N_{K/F}(K^\times)^s = \langle \zeta^{t_s} \rangle,
\end{equation}
and can identify this with the group $\mu_{n/t_s}$ of $n/t_s$th roots of unity.

Now as in \itemref{2}, let $\mathcal{D}_s \subset F^\times$ denote a set of coset representatives for the quotient
    $$
    F^\times / \mu_n\normKFim^s.
    $$

\begin{lemma} \label{L:FmunNormsrepresentatives}
With the notation of \itemref{1} and \itemref{2},  a set of coset representatives for
    $
    F^\times/\normKFim^s$
    is $\{\zeta^{i}d \mid 0\leq i < t_s, d\in \mathcal{D}_s\}$.
\end{lemma}

\begin{proof}
    The group generated by $\mu_n$ and $\normKFim^s$ is simply their product.  We note that
    $$
    \mu_n\normKFim^s/\normKFim^s \cong \mu_n/(\mu_n\cap \normKFim^s),
    $$
    which is finite.  Since $\mu_n\cap \normKFim^s=\langle \zeta^{t_s}\rangle$, it follows that this quotient is represented by the set $S=\{1,\zeta, \cdots, \zeta^{t_s-1}\}$.  Thus for each $d\in \mathcal{D}_s$ we have
    $$
    d\mu_n\normKFim^s = \bigsqcup_{0\leq i<t_s}d \zeta^i \normKFim^s,
    $$
    as required.
\end{proof}

The subgroup of \itemref{3} is more complex to define.

\begin{definition}\label{D:DeltatsI}
    Let $I = \{i_0<i_1<\cdots<i_k\}\in \mathcal{P}$ be nonempty and let  $t_s$ be as in \itemref{1}. If $k\geq 1$  then define the subgroup
$$
\Delta_{t_s,I} = \left\{ (\zeta^{\ell (i_1-i_0)}, \zeta^{\ell(i_2-i_0)}, \ldots, \zeta^{\ell (i_k-i_0)}) \in F^{[\times k]}  \mid \ell \in \frac{t_s}{\gcd(i_0,t_s)}\mathbb{Z}\right\}
$$
of $F^{[\times k]}$.
\end{definition}

This is a normal cyclic subgroup of $F^{[\times k]}$.

\begin{example}\label{Eg:munisnormtothes}
Suppose $\mu_n\cap N_{K/F}(K^\times)^s=\{1\}$
    and $\gcd(i_0,n)=1$ then $t_s=t_s/\gcd(i_0,t_s)=n$.  In this case $$\Delta_{t_s,I}=\{(1,1,\cdots,1)\in F^{[\times k]}\}.$$
    \end{example}
\begin{example}\label{Eg:otherextreme}
At the other extreme, if  $t_s\mid  i_0$,
then $\frac{t_s}{\gcd(i_0,t_s)}=1$, and  $\Delta_{t_s,I}$ is the cyclic group generated by $(\zeta^{i_1-i_0}, \cdots, \zeta^{i_k-i_0})$, which has order $$|\Delta_{t_s,I}|=n/\gcd(n,i_1-i_0, i_2-i_0, \cdots, i_{k}-i_0).$$  In this case, $\Delta_{t_s,I}$ is equal to the set denoted $\Delta_I$ in \cite[Definition 5.4]{NevinsPumpluen2025}.
\end{example}

For \itemref{4}, choose a set of coset representatives $\mathcal{E}_{s,I}$ for the quotient $F^{[\times k]}/\Delta_{t_s,I}$.  This allows us to define \itemref{5} as follows.

\begin{definition}\label{D:KtIa}
For each fixed choice of $a_{i_0}\in F^\times$, define the subset
$$
K(t_s,I;a_{i_0})=\left\{ a_{i_0}\beta^{i_0} + \sum_{j=1}^{k}a_{i_j}\beta^{i_j}  \middle| (a_{i_1},a_{i_2}, \ldots, a_{i_k})\in \mathcal{E}_{s,I} \right\} \subset K(I).
$$
\end{definition}

\begin{example}
In the setting of Example~\ref{Eg:munisnormtothes},
$K(t_s,I;a_{i_0})$ is simply the set of all elements of $K(I)$ whose $i_0$th coefficient is $a_{i_0}$.
\end{example}

Finally, we define \itemref{6}, which will be proven to parametrize the $F$-equivalence classes of algebras of the form $\Kone{a}$.

\begin{definition}\label{D:SsK}
Suppose $m=ns$ for some $s\in \ZZ$ and that $\zeta$ is a primitive $n$th root of unity in $F^\times$.  With the notation of \itemref{1}--\itemref{5}, we define
$$
\mathcal{S}_s(K)= \bigcup_{\substack{I\in \mathcal{P}\smallsetminus\{\emptyset\}\\I=\{i_0<i_1<\cdots<i_k\}}} \bigcup_{d\in \mathcal{D}_s} \bigcup_{\ell=0}^{\gcd(i_0,t_s)-1} K(t_s,I; d\zeta^\ell).
$$
\end{definition}

\begin{theorem} \label{Main:general}
Suppose $F$ contains a  primitive $n$th root of unity $\zeta$ and that $m=ns$ for some integer $s$.  Let ${t_s}>0$ be the least positive integer such that $\zeta^{t_s}\in N_{K/F}(K^\times)^s$.
 Then for fixed $K$ and $\sigma$, the distinct
 $F$-equivalence classes of the  Petit algebras

$$\mathbb{S}_{t^m-a}=K[t;\sigma]/K[t;\sigma](t^m-a) \quad (a\neq 0)$$
are represented exactly by the elements $a\in\mathcal{S}_s(K)$. That is, for every $b\in K^\times$ there exists a unique $a\in \mathcal{S}_s(K)$ such that
$\mathbb{S}_{t^m-b} \cong \mathbb{S}_{t^m-a}$
via an isomorphism of the form $G_{\tau,\alpha}$ for some $\tau\in \Aut(K/F)$ and $\alpha\in K^\times$.
\end{theorem}

\subsection{Proof of Theorem~\ref{Main:general}, and consequences}\label{SS:proof}

\begin{proof}[Proof of Theorem~\ref{Main:general}]
The objective is to implement Theorem~\ref{T:parametrizingnm}.  We first find a representatives for the set of $F$-equivalence classes, and then further quotient by the Galois group to obtain representatives for the isometry classes.
By
Lemma~\ref{IJdistinct}, it suffices to partition $K(I)$ into $F$-equivalence classes under $\simeq$ for each fixed nonempty $I\in \mathcal{P}$.  Therefore let us fix such an element   $I=\{i_0<i_1<\cdots<i_k\}$ for the remainder of the proof.

\emph{Our first step is to split the set $K(I)$ into equivalence classes with respect to multiplication by $\normKFim^s$.}  By Section \ref{sec:isometries}, at the level of Petit algebras, these correspond to isomorphisms of the form $G_{\id,\alpha}$ with $\alpha\in K^\times$.

Define for each $a_{i_0}\in F^\times$ the subset
$$
K(I)(a_{i_0}):=\{ a_{i_0}\beta^{i_0} + \sum_{j=1}^k a_{i_j}\beta^{i_j} | (a_{i_1},\ldots, a_{i_k})\in F^{[\times k]}\}\subset K(I)
$$
of elements of $K(I)$ whose first nonzero coefficient with respect to the ordered basis $B$ is $a_{i_0}$.  This is a partition of $K(I)$.

Multiplication by elements of $F^\times$ permutes the parts of this partition; in particular,  for every $\nu\in  \normKFim^s$ we have $\nu K(I)(a_{i_0})=K(I)(\nu a_{i_0})$.  Recall that by Lemma~\ref{L:FmunNormsrepresentatives}, each $a_{i_0}\in F^\times$ can be written uniquely as $d\zeta^i \nu$ for some $d\in \mathcal{D}_s$, $0\leq i < t_s$, and $\nu \in \normKFim^s$.  Thus for each $a\in K(I)(a_{i_0})$, we have $\nu^{-1}a \in K(I)(d\zeta_i)$, so $a\in b\normKFim^s$ for some $b\in K(I)(d\zeta_i)$.
Furthermore, no two elements of
$$
\bigcup_{d\in \mathcal{D}_s}\bigcup_{i=0}^{t_s-1}K(I)(d\zeta^i)
$$
are in the same coset of $K^\times/N_{K/F}(K^\times)^s$: if they lie in different sets $K(I)(d\zeta^i)$, then their leading terms  lie in different cosets of $F^\times/N_{K/F}(K^\times)^s$; and if they lie in the same set $K(I)(d\zeta^i)$, then as their first coefficients are equal, they cannot be $F^\times$-multiples of one another unless they are equal.

\emph{We now address the action of $\Gal(K/F)$, to account for all isometries, not just equivalences.}  We have already noted that $\Gal(K/F)$ preserves $K(I)$.  Given a subset of the form $K(I)(d\zeta^i)$, we first compute that for  every $k\in \bbZ$, we have
$$
\sigma^k(d\zeta^i\beta^{i_0}) = d\zeta^{i+ki_0}\beta^{i_0}.
$$
Therefore we infer that $\sigma^k(K(I)(d\zeta^i))=K(I)(d\zeta^{i+ki_0})$.  We next show there exists a minimal value $\ell$ such that $K(I)(d\zeta^{i+ki_0})\subseteq K(I)(d\zeta^\ell)\cdot \normKFim^s$ for every $k\in \ZZ$.

Set $\ell = i\mod \gcd(i_0,t_s)$; then $0\leq \ell < \gcd(i_0,t_s)$.  We first claim there exist integers $x,y$ such that $i+xi_0 = \ell+yt_s$.  Namely, choose integers $j,j'$ such that
$$
ji_0 + j't_s = \gcd(t_s,i_0).
$$
Then since $\ell-i$ is a multiple of $\gcd(t_s,i_0)$, there exists an integer $j''$ such that
$$
\ell - i = j''(ji_0 + j't_s)
$$
whence $\ell+(-j''j')t_s = i+(j''j)i_0$.  Thus taking $x=j''j$ and $y=-j''j'$ gives the desired equality $i+xi_0 = \ell+yt_s$.

For an arbitrary element $a\in K(I)(d\zeta^i)$, we have
\begin{align*}
    \sigma^x(a) &= d\zeta^{i+xi_0}\beta^{i_0} + \sum_{j=1}^k a_{i_j}\zeta^{xi_j}\beta^{i_j}\\
    &= d\zeta^{\ell+yt_s}\beta^{i_0} + \sum_{j=1}^k a_{i_j}\zeta^{xi_j}\beta^{i_j}\\
    &= (\zeta^{t_s})^y \left(d\zeta^{\ell}\beta^{i_0} + \sum_{j=1}^k a_{i_j}\zeta^{xi_j-yt_s}\beta^{i_j}\right) =: (\zeta^{t_s})^yb,
\end{align*}
where $b\in K(I)(d\zeta^\ell)$.
Thus since by definition $\zeta^{t_s}\in \normKFim^s$ and $\sigma^x\in\Gal(K/F)$, we deduce that $a\simeq b$ where  $b\in K(I)(d\zeta^\ell)$.

Thus we have shown that every $a\in K(I)$ is equivalent to an element of $K(I)(d\zeta^\ell)$ for some $d\in \mathcal{D}_s$ and $0\leq \ell<\gcd(i_0,t_s)$.

Moreover, elements of distinct such sets are not equivalent. Indeed, suppose
$a\in K(I)(d\zeta^{\ell})$ and $a'\in K(I)(d'\zeta^{\ell'})$ with $d,d'\in \mathcal{D}_s$,
$0\leq \ell,\ell'<\gcd(i_0,t_s)$, and $a\simeq a'$. Then $a'=\sigma^z(a)\nu$ for some
$z\in\mathbb{Z}$ and $\nu\in N_{K/F}(K^\times)^s$. Comparing the coefficients of $\beta^{i_0}$
gives
\[
d'\zeta^{\ell'}=d\zeta^{\ell+zi_0}\nu.
\]
Hence $d'\in d\,\mu_nN_{K/F}(K^\times)^s$, so $d=d'$ since $\mathcal{D}_s$ is a set of coset
representatives for $F^\times/\mu_nN_{K/F}(K^\times)^s$. It follows that
$\nu=\zeta^{\ell'-\ell-zi_0}\in\mu_n\cap N_{K/F}(K^\times)^s=\langle\zeta^{t_s}\rangle$. As $t_s\mid n$, this means $t_s\mid(\ell'-\ell-zi_0)$, so
$\ell'\equiv\ell \bmod \gcd(i_0,t_s)$, whence $\ell=\ell'$.

Fix an element $d\zeta^\ell$ with $d\in \mathcal{D}_s$ and $0\leq \ell < \gcd(i_0,t_s)$. It remains to establish when two elements $b,c\in K(I)(d\zeta^\ell)$ are
equivalent.  If $|I|=1$ then $K(I)(d\zeta^\ell)=\{d\zeta^\ell\beta^{i_0}\}$ is a singleton and there is nothing to show.  Otherwise, we have that $b\simeq c$ if and only if there exist $\nu\in \normKFim^s$ and $z\in \mathbb{Z}$ such that $c=\sigma^z(b)\nu$.  As $b$ and $c$ have the same first term $d\zeta^\ell\beta^{i_0}$, if $c=\sigma^z(b)\nu$ then  $$d\zeta^\ell\beta^{i_0}= \sigma^z(d\zeta^\ell\beta^{i_0})\nu = d\zeta^\ell\zeta^{zi_0}\nu\beta^{i_0}$$ so $\zeta^{zi_0}\nu=1$, that is, $\nu=\zeta^{-zi_0}$.  But since $\nu=\zeta^{-zi_0}\in \normKFim^s$ we must have that $zi_0$ is a multiple of $t_s$.  This happens if and only if  $z$ is a multiple of $t':=t_s/\gcd(i_0,t_s)$.  Therefore $b\simeq c$ if and only if there exists $z$ divisible by $t'$ such that $\sigma^z(b)\zeta^{-zi_0}=c$.
We compute explicitly that
$$
\sigma^z(b)\zeta^{-zi_0}
=\zeta^{-zi_0}\Big(d\zeta^{\ell+zi_0}\beta^{i_0}+\sum_{j=1}^k b_{i_j}\zeta^{zi_j}\beta^{i_j}\Big)
= d\zeta^{\ell}\beta^{i_0}+\sum_{j=1}^k b_{i_j}\zeta^{z(i_j-i_0)}\beta^{i_j}.
$$
We thus conclude that for two elements of $K(I)(d\zeta^\ell)$ we have
$$
d\zeta^{\ell}\beta^{i_0}+\sum_{j=1}^k b_{i_j}\beta^{i_j}\ \simeq\
d\zeta^{\ell}\beta^{i_0}+\sum_{j=1}^k c_{i_j}\beta^{i_j}
$$
if and only if we have the equality of vectors
$$
(c_{i_1},c_{i_2},\ldots, c_{i_k}) = (b_{i_1}\zeta^{z(i_1-i_0)}, b_{i_2}\zeta^{z(i_2-i_0)}, \ldots, b_{i_k}\zeta^{z(i_k-i_0)})
$$
for some $z\in t'\ZZ$, in other words, if and only if they are in the same coset of $\Delta_{t_s,I}$ (from Definition~\ref{D:DeltatsI}).
Consequently, choosing $(c_{i_1},\ldots,c_{i_k})$ to be an element of a set of representatives $\mathcal{E}_{s,I}$ for $F^{[\times k]}/\Delta_{t_s,I}$ gives a unique choice of $c\in K(t_s, I; d\zeta^\ell)$ equivalent to $b$,  and hence $c$ is the unique element of $\mathcal{S}_s(K)$ equivalent to $a$, as required.
\end{proof}

Observe from the proof that the Galois actions on the various parts $K(I)$, where $I=\{i_0<i_1<\cdots<i_k\}$ depends on the value of $i_0$.  For example, if $i_0=0$ then each $K(I)(d\zeta^i)$ is preserved by $\Gal(K/F)$ whereas if $i_0=1$ then $\Gal(K/F)$ acts transitively on the collection of subsets  $\{K(I)(d\zeta^j)\mid 0\leq j < t_s\}$ modulo $N_{K/F}(K^\times)^s$.  More generally the size of this orbit under the Galois action depends on $\gcd(i_0,t_s)$.

We can specialize our preceding result to the important special case of proper nonassociative cyclic algebras.

To state it succinctly, recall that the algebras are cyclic if $m=n$, that is, $s=1$.  Further, from Lemma~\ref{lemma:KminusFisP0}, the proper nonassociative Petit algebras are parametrized by $a\notin F$ so we restrict to $I\in \mathcal{P}_0=\mathcal{P}\smallsetminus \{\emptyset, \{0\}\}$.  Thus define $\mathcal{S}_1^\circ(K)=\mathcal{S}_1(K) \smallsetminus F$,
that is,
$$
\mathcal{S}_1^\circ(K) = \bigcup_{\substack{I\in \mathcal{P}_0\\I=\{i_0<i_1<\cdots<i_k\}}} \bigcup_{d\in \mathcal{D}_1} \bigcup_{\ell=0}^{\gcd(i_0,t_s)-1} K(t_s,I; d\zeta^\ell).
$$

\begin{corollary}\label{C:nonassociativecyclic}
    Suppose $m=n$ and that $F$ contains a primitive $n$th root of unity.  Then the distinct $F$-isomorphism classes of proper nonassociative cyclic algebras
    $$(K/F,\sigma,a)=
    \Kone{a}
    $$
    with $a\in K\smallsetminus F$ are represented by the triples $(K/F,\sigma,a)$ where $K$ is a cyclic Galois field extension of degree $n$, $\sigma$ is a generator of $\Gal(K/F)$, and $a\in \mathcal{S}_1^\circ(K)$.
\end{corollary}

\begin{proof}
    By Lemma \ref{Lem:differentextensions} and Theorem \ref{thm:2.3},
    if
    $$
    (K/F,\sigma,a)\cong (K'/F,\sigma',a')
    $$
    then we must have $K=K'$ and $\sigma=\sigma'$.  By \cite[Corollary 6.2]{NevinsPumpluen2025submitted},
    every  isomorphism between such algebras is an equivalence of the form $G_{\tau,\alpha}$.
    As noted in Lemma~\ref{lemma:KminusFisP0}, the algebras corresponding to $a\in F^\times$ correspond to equivalence classes in $K(\{0\})$.  The result now follows from Theorem~\ref{Main:general}.
\end{proof}

\begin{remark}\label{re:correction}
This corollary was previously known  only  when  $n=m$ is prime, and proved in \cite[Theorem 5.5]{NevinsPumpluen2025},  where there was a small mistake in the definition of the set $\mathcal{S}_1^\circ(K)$ for the case when $\zeta\not \in N_{K/F}(K^\times)$: in \cite[\S5]{NevinsPumpluen2025}, the set $\mathcal{D}_1$ of representatives of $F^\times/\mu_n\normKFim$ was omitted. While this is trivially the case in our setup when working with finite fields, it does not hold in general.
\end{remark}

In fact, although Theorem~\ref{Main:general} applies to Petit algebras of the form $\Kone{a}$, we can obtain some information about the existence of isomorphisms between \emph{general} Petit algebras, as a consequence of Corollary~\ref{cor:equivpolycyclic1} and Lemma~\ref{IJdistinct}, as follows.

\begin{corollary}\label{cor:generalcase}
 Suppose $F$ contains a primitive $n$th root of unity $\zeta$, and that $f(t)=t^m-\sum_{j=0}^{m-1}a_jt^j$ and $g(t) = t^m-\sum_{j=0}^{m-1}b_jt^j$ are elements of $K[t;\sigma]$ such that there exists  an isomorphism of  $F$-algebras $G_{\tau,\alpha}:\mathbb S_f \to \mathbb S_g$ for some $\tau\in \Aut(K/F)$ and $\alpha\in K^\times$.   For each $j$ let $I_j\in \mathcal{P}$ denote the subset such that $a_j\in K(I_j)$.  Then for each $j$ such that $n\mid (m-j)$ and $a_j\neq 0$, we have $b_j\in K(I_j)$.
\end{corollary}

\begin{proof}
    By Corollary~\ref{cor:equivpolycyclic1}, the existence of such an isomorphism $G_{\tau,\alpha}:\mathbb S_f\to \mathbb S_g$ implies that for each $j$ with $(a_j,b_j)\neq (0,0)$ we have an isomorphism
    $$
    G_{\tau,\sigma^j(\alpha)}:\mathbb S_{t^{m-j}-a_j} \to \mathbb S_{t^{m-j}-b_j}.
    $$
    In particular, we have that $a_j\simeq b_j$.  When $n\mid m-j$, we are in the setting that $N^\sigma_{m-j}(K^\times)\subset F^\times$ and thus $a_j\in K(I_j)$ if and only if $b_j\in K(I_j)$.
\end{proof}

\section{Application to local nonarchimedean fields}  \label{sec:local}
In this section, we demonstrate how to evaluate the theory developed in Section~\ref{SS:parametrization}
explicitly for a given family of fields, namely where  $F$ is a local nonarchimedean field.  We first give a concise introduction to these fields  in Section~\ref{sec:p-adic fields}, then prove some needed results about Kummer extensions in Section~\ref{SS:Kummerpadic}, and then carry out the parametrization in Section~\ref{SS:pametrizationpadic}.

\subsection{Local nonarchimedean  fields} \label{sec:p-adic fields}

A  locally compact, complete, nondiscrete field equipped with an ultrametric norm $\vert \cdot \vert$ is called a \emph{local nonarchimedean field}.  For each prime number $p$, there are precisely two classes of such fields \cite[Chapter II]{Lang1994}:  the finite algebraic extensions of the $p$-adic numbers $\Q_p$, which are the completion of a number field with respect to the $p$-adic norm and have characteristic zero; and the fields $\F_q(\!(t)\!)$ of Laurent series over a finite field of order $q=p^\gamma$ for some $\gamma\geq 1$, which have characteristic $p$.

Each  local nonarchimedean field $F$ has a maximal compact subring, called its integer ring, given by  $\RR=\{x\in F : |x|\leq 1\}$.  This is a local ring with unique maximal ideal $\PP = \{x \in F : |x|<1\}$.  Let $\varpi$ denote a generator of $\PP$, that is, $\PP=\varpi \RR$. The quotient ring $\mathfrak{f} :=\RR/\PP$  is called the residue field.  It is a finite field of order $q=p^\gamma$ for some $\gamma\in \N$, and its characteristic $p$ is called the residual characteristic of $F$.

The valuation of an element $a\in F^\times$ is normalized so that $\val(a)=\max\{n\in \Z: a\in \PP^n\}$ and we scale the norm so that $|a| = q^{-\val(a)}$.  With this normalization we have $\val(\varpi)=1$.

The following general version of Hensel's Lemma is adapted from \cite[Thm 7.3]{Eisenbud1995}.

\begin{lemma}[Hensel's Lemma] \label{le:Hensel}
Let $f(x)\in \RR[x]$ be a polynomial and $f'(x)$ its formal derivative.  If $a\in \RR$ satisfies
$$
f(a) \equiv 0 \mod f'(a)^2\PP,
$$
then there exists a unique $b\in \RR$ satisfying
$$
f(b)=0 \quad \text{and}\quad b\equiv a \mod f'(a)\PP.
$$
\end{lemma}

As an important application, note that the elements of $\mathfrak{f}^\times$ are the roots of $f(x)=x^{q-1}-1$.  By Hensel's Lemma, each of these corresponds to a unique root of $f(x)$ in $\RR$, which we call the Teichm\"uller lifts.  They form the unique subgroup $\mu_{q-1}$ of $(q-1)$th roots of unity in $F^\times$.  (Recall that this subgroup contains $\mu_n$ for any $n|(q-1)$.)

It follows that we may write $F^\times$ as the direct product of groups
\begin{equation}\label{E:productform}
F^\times \cong \mu_{q-1} \times (1+\PP) \times \Z
\end{equation}
via the map that associates to a triple $(b, 1+u\varpi, n)$  the element $b(1+u\varpi) \varpi^n$.

Now let  $K$ be a finite algebraic extension of $F$ of degree $n$.  This is again a local nonarchimedean field. Let $\mathfrak{k}$ be its residue field; it is an extension of $\mathfrak{f}$ of some degree $f$ called the inertial degree. Let $\varpi_K$ be a uniformizer of $K$; then if we normalize the valuation on $K$ so that $\val(\varpi)=1$ then $\val(\varpi_K)=1/e$ for some positive integer $e$ called the ramification degree.  We have $n=ef$.

The extension is called unramified if $e=1$ (in which case one may take $\varpi_K=\varpi$) and totally ramified if $f=1$ (in which case one may identify $\mathfrak{f}=\mathfrak{k}$).  Between these extremes we say the extension is ramified.  When $\gcd(e,p)=1$ the ramification is called tame (and wild otherwise).  The maximal unramified subextension $E$ of $F$ in $K$ has degree $f$ over $F$ and $K/E$ is totally ramified.

Each totally and tamely ramified extension $K$ of $F$ of degree $e$ is obtained by adjoining to $F$ a root of $x^e-u\varpi$, for some choice of $u\in \RR^\times/(\RR^\times)^e$. These extensions are Galois only if $F$ contains a primitive $e$th root of unity.

By local class field theory \cite[XI \S4]{Lang1994}, the abelian Galois extensions of degree $n$ of $F$ are in one-to-one correspondence with the subgroups of index $n$ of $F^\times$, via the map that assigns $K/F$ to the subgroup $\normKFim$, which is the image of the norm map.  In particular, this implies that $F^\times/\normKFim$ is a finite group of order $n$ equal to the degree of the extension.

For example, if $K/F$ is unramified of degree $n$ then $\normKF(\RR_K^\times)=\RR^\times$ and  $F^\times/\normKFim$ is represented by the elements
$$
\{ 1, \varpi, \varpi^2, \cdots, \varpi^{n-1} \}.
$$
If instead $K/F$ is totally ramified of degree $n$ with $\gcd(n,p)=1$, then  $F^\times/\normKFim \cong \RR^\times/(\RR^\times)^n \cong \mathfrak{f}^\times/ \mathfrak{f}^{\times n}$.

\subsection{Kummer extensions of local nonarchimedean fields}\label{SS:Kummerpadic}

Let $F$ be a local nonarchimedean field of residual characteristic $p$ and residue field $\mathfrak{f}$ of order $q=p^\gamma$.  Let $\prootF$ denote a generator of the Teichm\"uller subgroup $\mu_{q-1}\cong \mathfrak{f}^\times$.

  Let $n$ be a positive integer such that $n|(q-1)$.  Then $\prootF^{(q-1)/n}\in F^\times$ is a primitive $n$th root of unity, and all cyclic extensions $K$ of degree $n$ of $F$ are given by Kummer theory.  Note that since $p\nmid n$, each extension of degree $n$ is separable and tamely ramified.

\begin{lemma}\label{lemma:imageofnormmappadic}
    Suppose $n\mid (q-1)$ and $K/F$ is an extension of degree $n=ef$, where $e$ is its degree of ramification and $f$ its inertial degree. Then we have
    $$K=F(\prootK, (\prootK^\ell\varpi)^{1/e})$$
    for some generator $\prootK$ of $\mu_{q^f-1}\subset K^\times$ and integer $\ell$ satisfying $0\leq \ell < e$.   Moreover, there is a choice of generator $\prootF$ of $\mu_{q-1}\subset F^\times$ such that
    $$
    \normKFim = \langle \prootF^e,
    1+\PP_F, (-1)^{f(e-1)}\prootF^\ell\varpi^f\rangle.
    $$
\end{lemma}

\begin{proof}
    Let $K$ be an extension of $F$ of degree $n$. Let $E$ denote the maximal unramified extension of $F$ in $K$; then $K/E$ is a totally ramified extension of degree $e$ and $E/F$ is an unramified extension of degree $f$.

    Recall that we can factor the group $K^\times$ as $\mu_{q^f-1} \times (1+\PP_K) \times \langle \varpi_K\rangle$.  We compute the image of the norm map $\normKF=\normEF\circ\normKE$ by determining its image on each of the generating subgroups $1+\PP_K$, $\mu_{q^f-1}$ and $\langle \varpi_K \rangle$ separately.

    First note that since $\gcd(n,p)=1$, Hensel's Lemma gives that every element of $1+\PP_F$ is an $n$th power of an element of $1+\PP_F$.  Since the norm of an element of $1+\PP_K$ must be in $1+\PP_F$, we deduce  $\normKF(1+\PP_K)=1+\PP_F$.

    Next, since $K/E$ is totally ramified, the residue fields of $K$ and $E$ are isomorphic; since $E/F$ has degree $f$, these residue fields have order $q^f$.  Since $\Gal(K/E)$ induces the trivial action of this residue field, it is also trivial on its Teichm\"uller subgroup $\mu_{q^f-1}$.  It follows that the restriction of $\normKE$ to $\mu_{q^f-1}$ is the $e$-power map.  Let  $\prootK$ be a generator of $\mu_{q^f-1}$; then we have $E=F(\prootK)$.   By construction, $\normEF(\prootK) = \prootK^{(q^f-1)/(q-1)}$ generates $\mu_{q-1}$; without loss of generality we may set $\prootF=\prootK^{(q^f-1)/(q-1)}$. Thus we have
    $$
    \normKF(\prootK) = \normEF(\normKE(\prootK))=\normEF(\prootK^e)=\normEF(\prootK)^e.
    $$
Since $e|(q-1)$, we deduce that  $\normKF(\mu_{q^f-1})=\mu_{(q-1)/e}$.

    Finally, we compute $\normKF(\varpi_K)$.  Since $E/F$ is unramified, a uniformizer $\varpi$ of $F$ is also a uniformizer of $E$.  As recapped in \cite[\S2.3]{ChinnerThesis2021}, for example, the totally ramified extensions $K/E$ are given by adjoining a root of $x^e-r\varpi$ for
    some choice of $r\in \intR_E^\times/(\intR_E^{\times})^e$.   Let $\varpi_K=(r\varpi)^{1/e}$ be one root and let $\zeta_e=\prootK^{(q^f-1)/e}$ denote a primitive $e$th root of unity.  Then the distinct roots of this polynomial are $\{\zeta_e^i (r\varpi)^{1/e}\mid 0\leq i <e\}$ whence
    $$
    \normKE(\varpi_K) = \prod_{i=0}^{e-1}(\zeta_e^i (r\varpi)^{1/e})= \zeta_e^{e(e-1)/2} r\varpi = \begin{cases}
        r\varpi & \text{if $e$ is odd}\\
        -r\varpi & \text{if $e$ is even.}
    \end{cases}
    $$
    This sign $\pm$ can be recorded as $(-1)^{(e-1)}$.
    Since $E/F$ is unramified of degree $f$, we have $\normEF(\varpi)=\varpi^f$ and thus
    $$
    \normKF(\varpi_K)=\normEF((-1)^{(e-1)}r\varpi) = (-1)^{f(e-1)}\normEF(r)\varpi^f.
    $$
    This holds for any $r$.  Since $\gcd(e,p)=1$ we may apply Hensel's lemma to deduce that the distinct totally ramified extensions $K/E$ are parametrized by $r \in \mu_{q^f-1}/(\mu_{q^f-1})^e$.  Since $e\mid (q^f-1)$, this quotient is represented by $\{\prootK^\ell\mid 0\leq \ell < e\}$.
    If $r=\prootK^\ell$ then $\normEF(r)=r^{(q^f-1)/(q-1)} = \prootF^\ell$, as required.

    Note that in this notation, we have
    $$K=E((r\varpi)^{1/e})=F(\prootK)((\prootK^\ell\varpi)^{1/e})$$
    where $\prootF=\prootK^{(q^f-1)/(q-1)}$ is a generator of $\mu_{q^f-1}\subset K^\times$ and $0\leq \ell < e$.
\end{proof}

Note that in almost all cases, we have $(-1)^{f(e-1)}\in \mu_{q-1}$, so that  $(-1)^{f(e-1)}\prootF^\ell\varpi^f=\prootF^{\ell+c}\varpi^f$ for some $c\in \{0,\frac12(q-1)\}$.

Note that under our global assumption $n\mid(q-1)$ we always have
$(-1)^{f(e-1)}\in\mu_{q-1}$: if $p$ is odd then $-1\in\mu_{q-1}$, and if $p=2$ then $q-1$
is odd, so $e\mid n$ is odd and $(-1)^{f(e-1)}=1$. Hence
$(-1)^{f(e-1)}\xi_0^\ell\varpi^f=\xi_0^{\ell+c}\varpi^f$ for some
$c\in\{0,\tfrac12(q-1)\}$, where $c=0$ if $p=2$.

\subsection{Explicit parametrization results for local nonarchimedean fields}\label{SS:pametrizationpadic}

We now apply the theory developed in Section~\ref{SS:parametrization} of parametrize the distinct $F$-equivalence classes of algebras of the form $\Kone{a}$ where $K$ is a cyclic Kummer extension of a local nonarchimedean field $F$, where $[K:F]=n=ef$, with $\Gal(K/F)=\langle \sigma \rangle$, under the assumption that $m$ is relatively prime to $p$ and divisible by $n$.
Note that then $s=m/n$ is relatively prime to $p$.   Thus we infer from Section~\ref{SS:Kummerpadic} that
\begin{equation}\label{E:padicimageofnormmaptothes}
    \normKFim^s = \langle \prootF^{es},
    1+\PP_F, \prootF^{s(\ell+c)}\varpi^{fs}\rangle,
\end{equation}
where $c$ is defined by $\prootF^c = (-1)^{f(e-1)}$.

We now follow the steps of Section~\ref{SS:parametrization}.

\itemref{1}: \emph{A primitive $n$th root of unity $\zeta$, and the size $n/t_s$ of the intersection $\mu_n\cap \normKFim^s$.}

Since $n|(q-1)$ we may take $\zeta = \prootF^{(q-1)/n}$, or equivalently (when $\prootK$ and $\prootF$ are chosen compatibly, as in Lemma~\ref{lemma:imageofnormmappadic}) $\zeta = \prootK^{(q^f-1)/n}$.

\begin{lemma}\label{Lemma:tsandmunNorm}
    The intersection $\mu_n\cap \normKFim^s$ has $\displaystyle \frac{\gcd(nes,q-1)}{\gcd(es,q-1)}$ elements, so we have $t_s = \displaystyle \frac{n\gcd(es,q-1)}{\gcd(nes,q-1)}$ and
    $$
    \mu_n\normKFim^s = \langle \prootF^{\gcd(es,(q-1)/n)}, 1+\PP_F, \prootF^{s(\ell+c)}\varpi^{fs}\rangle.
    $$
\end{lemma}

\begin{proof}
    Since $n|(q-1)$, we have $\mu_n\subset \mu_{q-1}$.
    If $a\in \mu_n\cap \normKFim^s$ then by
    \eqref{E:padicimageofnormmaptothes}, there exist $b\in \langle \prootF^{es}\rangle$, $u\in \RR$ and $j\in \ZZ$ such that $a=b(1+u\varpi)(\prootF^{s(\ell+c)}\varpi^{fs})^{j}$.  Since $\val(a)=0$, we deduce that $j=0$, and since $a$ has order dividing $q-1$, $u=0$.  Therefore $\mu_n \cap \normKFim^s\subset \langle \prootF^{es}\rangle$, whence  $$\mu_n\cap \normKFim^s = \mu_n \cap \langle \prootF^{es}\rangle = \mu_n \cap \mu_{(q-1)/\gcd(es,q-1)}.$$
    Thus the intersection has order $\gcd\left(n,\frac{q-1}{\gcd(es,q-1)}\right)$.  We simplify this as
    $$\gcd\left(n,\frac{q-1}{\gcd(es,q-1)}\right)=
    \frac{\gcd(n\gcd(es,q-1),q-1)}{\gcd(es,q-1)}
    = \frac{\gcd(nes,q-1)}{\gcd(es,q-1)}.
    $$
    By definition of $t_s$, this order is equal to $n/t_s$.
    The final statement follows by noting that the group generated by $\prootF^{es}$ is also generated by $\prootF^{\gcd(es,q-1)}$ and that the group generated by $\prootF^{(q-1)/n}$ and $\prootF^{\gcd(es,q-1)}$ is generated by $\prootF^{\gcd((q-1)/n,es,q-1)}=\prootF^{\gcd(es, (q-1)/n)}$.
\end{proof}

\itemref{2}: \emph{A set of representatives $\mathcal{D}_s$ of the quotient $F^\times/\mu_n\normKFim^s$ (Lemma~\ref{L:FmunNormsrepresentatives}).}

\begin{proposition}\label{Proposition:Dspadic}
Let $K/F$ be cyclic of degree $n=ef$ such that $n\mid q-1$ and  $s=m/n$ with $\gcd(s,p)=1$.
Then a set of representatives for $F^\times/\mu_nN_{K/F}(K^\times)^s$ is given by
$$
\mathcal{D}_s=\left\{\prootF^{i}\varpi^{j}\, |\, 0\leq i<\gcd\left(es,\frac{q-1}{n}\right),\ 0\leq j<fs\right\},
$$
which has order $|\mathcal{D}_s|=fs \gcd\left(es,\frac{q-1}{n}\right)$.
\end{proposition}

Note that the representative set $\mathcal D_s$ depends only on $e$, $f$, $s$, and $q$, and is independent of the sign $(-1)^{f(e-1)}$ or the particular choice of ramified extension given by the choice of $\ell$.

\begin{proof}
Let us set $g_s=\gcd(es,(q-1)/n))$ for brevity.
By Lemma~\ref{Lemma:tsandmunNorm}, we have $H:=\mu_n\normKFim^s = \langle \prootF^{g_s}, 1+\PP_F, \prootF^{s(\ell+c)}\varpi^{fs}\rangle$.  Then every element $h\in H$ can be written in the form
\begin{equation}\label{E:expressionforh}
h= a(1+u\varpi)(\prootF^{s(\ell+c)}\varpi^{fs})^k
\end{equation}
for some $a\in \langle \prootF^{g_s} \rangle$, $u\in \RR$ and $k\in \ZZ$.

We first show that the elements of $\mathcal{D}_s$ lie in distinct cosets.
Namely, if $\prootF^{i}\varpi^{j} \in \prootF^{i'}\varpi^{j'}H$ for some $(i,j),(i',j')$ satisfying $0\leq i,i'< g_s$ and $0\leq j,j'<fs$, then we have $\prootF^{i-i'}\varpi^{j-j'} =h \in H$. Writing $h$ as in \eqref{E:expressionforh}, we deduce by comparing valuations that $fsk=j-j'$.  As $|j-j'|<fs$, this implies $k=0$ and $j=j'$.  Since $\prootF^{i-i'}\in \mu_{q-1}$, which intersects $1+\PP_F$ trivially, this forces $u=0$ and $a=\prootF^{i-i'}$ so that $g_s|(i-i')$, whence $i=i'$.  Thus the cosets are distinct.

We next show every element of $F^\times$ lies in one of these cosets.  Let $x\in F^\times$; then there exist $t\in \ZZ$ and $u\in \RR$ such that $x=\prootF^t(1+u\varpi)\varpi^{\val(x)}$.  Set $j:=\val(x)\mod fs$ so that $0\leq j <fs$ and there exists $k\in \ZZ$ such that $\val(x)-j=fsk$; and set $i := t-sk(\ell+c) \mod g_s$, so that $0\leq i < g_s$ and there exists $t'\in \ZZ$ such that $t-i = sk(\ell+c)+t'g_s$.
Then we have
$$
\prootF^{-i}\varpi^{-j}x  (1+u\varpi)\varpi^{\val(x)-j}
=  \prootF^{s(\ell+c)k+t'g_s}(1+u\varpi)\varpi^{fsk}
= (\prootF^{g_s})^{t'}(1+u\varpi)(\prootF^{s(\ell +c)}\varpi^{fs})^k \in H.
$$
\end{proof}

\itemref{3}: \emph{For each $I\in \mathcal{P}$, with $|I|=k+1\geq 2$, a cyclic group $\Delta_{t_s,I}$ (Definition~\ref{D:DeltatsI}) of $F^{[\times k]}$.}

This is the subgroup
\begin{equation}\label{E:DeltatsIforpadic}
\Delta_{t_s,I} = \langle (\zeta^{t' (i_1-i_0)}, \zeta^{t'(i_2-i_0)},\ldots,  \zeta^{t'(i_k-i_0)}) \rangle \subset F^{[\times k]}
\end{equation}
where $t'=\frac{t_s}{\gcd(i_0,t_s)}$.

\begin{lemma}\label{le:4.5}
    We have $|\Delta_{t_s,I}|=\displaystyle   \frac{n/t'}{\gcd(i_1-i_0, i_2-i_0,\cdots, i_k-i_0,n/t')}.$
\end{lemma}

\begin{proof}
    The order of $(\zeta^{t' (i_1-i_0)}, \zeta^{t'(i_2-i_0)},\ldots,  \zeta^{t'(i_k-i_0)})$ is equal to that of
    $$
    \zeta^{\gcd(t'(i_1-i_0), t'(i_2-i_0), \ldots, t'(i_k-i_0), n)}.
    $$
    Since $t_s$ divides $n$, so does $t'=t_s/\gcd(i_0,t_s)$, and therefore
    $\gcd(t'(i_1-i_0), t'(i_2-i_0), \ldots, t'(i_k-i_0), n)=t'\gcd(i_1-i_0, i_2-i_0, \ldots, i_k-i_0, n/t')$. The formula follows.
\end{proof}

\itemref{4}: \emph{A set of representatives $\mathcal{E}_{s,I}$ for the quotient $F^{[\times k]}/\Delta_{t_s,I}$.}

Let us first give, in Algorithm~\ref{Algorithm:1}, a general algorithm for computing a set of representatives for $G/H$ where $G \cong (\ZZ/N\ZZ)^k$ and $H$ is a cyclic subgroup of $G$.

\begin{algorithm}
\caption{Generating a set of representatives for $(\ZZ/N\ZZ)^k/H$ where $H$ is cyclic. \label{Algorithm:1}}
\begin{algorithmic}[1]
\Input : $N,k\in \ZZ_{\geq 1}$,
$G = \ZZ/N\ZZ^k$, $a_1,a_2,\cdots, a_k\in \ZZ$, $H=\langle (a_1,\cdots, a_k)\rangle \subset G$
\EndInput
\Output : A set $E \subset G$ of representatives for $G/H$
\EndOutput
\State $M_1 \gets N/\gcd(a_1,N)$ \Comment{The order of the first term in $\ZZ/N\ZZ$}
\State $M\gets M_1$
\For{$i$ in $2$ to $k$}
    \State $H_i \gets H^M=\langle (0, \cdots, 0, M a_i, M a_{i+1},\cdots, M a_k)$
    \State $M_i \gets N/\gcd(Ma_i,N)$
    \Comment{The order of term $i$ in $H_i$}
    \State $M \gets M\times M_i$ \Comment{Accumulated order}
\EndFor
\State \Return $E \gets \{(c_1,\cdots,c_k) \mid \forall i, 0\leq c_i < N/M_i \}$
\end{algorithmic}
\end{algorithm}

Now observe that $F^{[\times k]} \cong \mu_{q-1}^k\times (1+\PP_F)^k \times \ZZ^k$ and that $\Delta_{t_s,I}\subset \mu_{q-1}^k\times \{1\}\times \{0\}$, so that, using \eqref{E:DeltatsIforpadic}, we have
\[
F^{[\times k]}/\Delta_{t_s,I}\cong
\mu_{q-1}^k/\langle(\zeta^{t'(i_1-i_0)},\zeta^{t'(i_2-i_0)},\dots,\zeta^{t'(i_k-i_0)})\rangle
\times (1+\mathcal{P}_F)^k\times \mathbb{Z}^k,
\]
where $t'=t_s/\gcd(i_0,t_s)$.

\color{black}

\begin{lemma}\label{lemma:Estpadic}
Let $E$ be the output of Algorithm~\ref{Algorithm:1}, upon input $N=q-1$, $k$,
    $a_j=t'(i_j-i_0)\frac{q-1}{n}$
    for $1\leq j\leq k$.  Then
    $$
E_{s,I}=\{(\xi_0^{c_1}x_1,\xi_0^{c_2}x_2,\dots,\xi_0^{c_k}x_k)\mid (c_1,\dots,c_k)\in E,\
x_1,\dots,x_k\in (1+\mathcal{P}_F)\langle\varpi\rangle\}
$$
is a set of representatives for the quotient $F^{[\times k]}/\Delta_{t_s,I}$.
\end{lemma}

\begin{proof}
Let $N=q-1$.   The map sending $c\in \ZZ/N\ZZ$ to $\zeta_0^c\in \mu_{q-1}$ is an isomorphism,
 and since $\zeta=\xi_0^{(q-1)/n}$, the image of the subgroup $\Delta_{t_s,I}$ is the
cyclic subgroup $H=\langle(a_1,\dots,a_k)\rangle$ with $a_j=t'(i_j-i_0)\frac{q-1}{n}$. By
\eqref{E:productform}, $F^\times=\mu_{q-1}\times(1+\mathcal{P}_F)\times \langle\varpi\rangle$, and
$\Delta_{t_s,I}$ lies in the factor $\mu_{q-1}^k$ of $F^{[\times k]}$.

Thus it suffices to prove that Algorithm~\ref{Algorithm:1} is correct, that is, that its output $E$ is in bijective correspondence with the cosets of $G/H$.  Adopt the notation of the algorithm.

The algorithm produces a set $E$ with  $M=\prod_{i=1}^kM_i$ elements.  Noting that
    $$
    M_i = \frac{N}{\gcd(\prod_{j=1}^{i-1}M_ja_i,N)} = \frac{N/\prod_{j=1}^{i-1}M_j}{\gcd(a_i,N/\prod_{j=1}^{i-1}M_j)},
    $$
    it is direct to prove by induction that
    $$
    M_i = \frac{\gcd(a_1,\cdots, a_{i-1},N)}{\gcd(a_1,\cdots, a_{i},N)},
    $$
    so that the total order $M=\prod_{j=1}^kM_j$ of $H$ is $M=N/\gcd(a_1,\cdots, a_k,N)$, as expected.  The number of elements of $E$ is
    $$
    \prod_{j=1}^k\frac{N}{M_j} = \frac{N^k}{M},
    $$
    which agrees with the number of cosets of $H$ in $G$.

    Now suppose that two elements $c=(c_1,\cdots, c_k), c'=(c_1',\cdots, c_k')\in E$ lie in the same coset of $H$, that is, $(c_1-c_1', \cdots, c_k-c_k')\in H$. Let $i$ be the index of the first nonzero term $c_i-c_i'$, if it exists.  Then since $\tilde{M}_i=\prod_{j<i} M_i$
    is the order of the projection of $H$ onto its first $i-1$ factors, we deduce that $c-c'\in H^{\tilde{M}_i}$, whence $c_i-c_i'\in \gcd(\tilde{M}_i a_i, N)\ZZ$; but this is impossible as $0\leq c_i,c_i'< N/M_i=\gcd(\tilde{M}_i a_i, N)$.  Thus $c=c'$, whence $E$ parametrizes the cosets of $H$ in $G$.
\end{proof}

\itemref{5}: \emph{Subsets $K(t_s,I;a_{i_0})$ of $K(I)$} are defined as in Definition~\ref{D:KtIa}, where $\mathcal{E}_{s,I}$ is as in Lemma~\ref{lemma:Estpadic}.

\itemref{6}: \emph{The parametrizing set $\mathcal{S}_s(K)$}.

Using the preceding, we can simplify the expression for $\mathcal{S}_s(K)$, as we give in our final summarizing theorem.

\begin{theorem} \label{Main:padic}
Suppose $F$ is a nonarchimedean local field with residual characteristic $p$ and residue field of order $q$ such that $\mu_{q-1}=\langle \prootF\rangle$.  Let $K/F$ be a cyclic extension of degree $n=ef$ where $e$ is its ramification degree, such that $n|(q-1)$ and $(-1)^{f(e-1)}\in \mu_{q-1}$, and suppose $\Gal(K/F)=\langle \sigma \rangle$.  Let  $m\in \ZZ$ such that $n|m$ and $\gcd(m,p)=1$ and set $s=m/n$.  Let $\gcdesn=\gcd(es,\frac{q-1}{n})$, $\Thetaes=\gcd(es,q-1)$ and $L=\gcd(i_0, \Thetaes/\gcdesn)$.

 Then for every $b\in K^\times$ there exists a unique $a\in \mathcal{S}_s(K)$ where
 $$
    \mathcal{S}_s(K)= \bigcup_{\substack{I\in \mathcal{P}\smallsetminus\{\emptyset\}\\I=\{i_0<i_1<\cdots<i_k\}}}
    \bigcup_{i=0}^{\gcdesn-1}
    \bigcup_{j=0}^{fs-1}
\bigcup_{\ell=0}^{L-1} K\left(\tfrac{[es]_{q-1}}{[es]_{\frac{q-1}{n}}}, I; \xi_0^{\,i+\frac{\ell(q-1)}{n}}\varpi^{j}\right)
$$
such that $\mathbb{S}_{t^m-b}$
is isomorphic to
$\mathbb{S}_{t^m-a}$
via an isomorphism of the form $G_{\tau,\alpha}$ for some $\tau\in \Gal(K/F)$ and $\alpha\in K^\times$.
\end{theorem}

\begin{proof}
    By Theorem~\ref{Main:general}, these classes are parametrized by the set $\mathcal{S}_s(K)$ of Definition~\ref{D:SsK}.  By Lemma~\ref{Lemma:tsandmunNorm} we have $t_s=\Thetaes/\gcdesn$ so
    $0\leq \ell < \gcd(i_0,t_s)=\gcd(i_0, \Thetaes/\gcdesn)=L$.
    By Proposition~\ref{Proposition:Dspadic},
    $\mathcal{D}_s = \{\prootF^i\varpi^j \mid 0\leq i < \gcdesn, 0\leq j < fs\}$, so the union over $d\in \mathcal{D}_s$ may be replaced by a union over $i$ and $j$ separately.  Finally, as identified before Lemma~\ref{Lemma:tsandmunNorm}, we may take $\zeta = \prootF^{(q-1)/n}$ so the expression $d\zeta^\ell$ simplifies to  $\prootF^{i+\frac{\ell(q-1)}{n}}\varpi^j$.  Thus the set in the statement of the theorem is $\mathcal{S}_s(K)$, as claimed.
\end{proof}

\subsection{An application to number fields}\label{sec:number fields}

Let $F$ be a number field that is, a finite algebraic extension of $\mathbb{Q}$) containing a primitive $n$th root of unity $\zeta$ and let $K/F$ be
cyclic of degree $n$ with $\Gal(K/F)=\langle\sigma\rangle$.

Let $v$ be a finite place of $F$ such that there is exactly one place $w$ of $K$ above $v$. Let $K_w$ and $F_v$ denote the corresponding completions of $K$ and $F$, respectively.  If $p\in \ZZ$ is the prime number such that $v|p$ then $K_w$ and $F_v$ are $p$-adic fields of residual characteristic $p$.  Then $K_w=K\otimes_F F_v$ is a cyclic field extension of $F_v$ of degree $n$ and  $\sigma$ extends uniquely to a generator $\sigma_w=\sigma \otimes \id_{F_v}$ of $\mathrm{Gal}(K_w/F_v)$.
These maps induce a
canonical isomorphism
\[
(K[t;\sigma]/K[t;\sigma](t^m-a))\otimes_F F_v\cong
K_w[t;\sigma_w]/K_w[t;\sigma_w](t^m-a).
\]
of  $F_v$-algebras, for any $a\in K$.

For any $\tau\in \Gal(K/F)$, let  $\tau_w=\tau \otimes \id_{F_v}$.

\begin{proposition}\label{P:numberfields}
Let $m=ns$ and $a\in K^\times$.
\begin{enumerate}[(i)]
\item If  $a\simeq_F b$ then $a\simeq_{F_v} b$.
\item If  $v\nmid m$ and $a\simeq_F b$, then $a$ and $b$ are represented by
the same element of $\mathcal{S}_s(K_w)$.
\end{enumerate}
\end{proposition}

\begin{proof}
$(i)$ Suppose $G_{\tau,\alpha}: K[t;\sigma]/K[t;\sigma](t^m-a)\to
K[t;\sigma]/K[t;\sigma](t^m-b)$ is an $F$-equivalence.  Then $G_{\tau,\alpha}\otimes\mathrm{id}_{F_v}$ is an isomorphism of $F_v$-algebras, and it maps $ct^i$ to $\tau_w(c)N^{\sigma_w}_i(\alpha)t^i$
for all $c\in K_w$. Hence this coincides with the map
$G_{\tau_w,\alpha}:K_w[t;\sigma_w]/K_w[t;\sigma_w](t^m-a)\to
K_w[t;\sigma_w]/K_w[t;\sigma_w](t^m-b)$, which is thus an $F_v$-equivalence.

$(ii)$  If  $v\nmid m$, then $K_w/F_v$ satisfies the hypotheses of
Theorem~\ref{Main:padic}.
 Let $p$ be the residual characteristic of $F_v$ and $q_v$ the order of its residue
field. As $n\mid m$ and $v\nmid m$, we have $p\nmid m$ and $p\nmid n$. Since
$\mu_n\subset F\subset F_v$ and $p\nmid n$, the subgroup $\mu_n$ lies in the Teichm\"uller
subgroup $\mu_{q_v-1} \subset F^\times_v$, so $n\mid(q_v-1)$. The  assertion now follows from $(i)$ and
Theorem~\ref{Main:padic}.
\end{proof}

Indeed we know more.  Suppose the Kummer extension $K/F$ is given by $K=F(\beta)$ with $\beta\in K$ as in Section~\ref{SS:Kummer}.  Then for every $(v,w)$ as above, the Kummer basis
$\{1,\beta,\dots,\beta^{n-1}\}$ of $K/F$ is also one of $K_w/F_v$, satisfying $\sigma_w(\beta)=\sigma(\beta)=\zeta\beta$. Note further that for each $I\in \mathcal{P}$ we have $K(I)\subset K_w(I)$.

Now suppose that $a,b\in K\smallsetminus F$ satisfy $a\simeq_F b$. Then by Lemma~\ref{IJdistinct} there exists some $I=\{i_0<i_1<\cdots<i_k\}\in \mathcal{P}$ such that $a,b\in K(I)$.  Write
$a=\sum_{j=0}^k a_{i_j}\beta^{i_j}$ and $b=\sum_{j=0}^k b_{i_j}\beta^{i_j}$ in $K(I)$, for some $a_{i_j},b_{i_j}\in F^\times$.
Then
Proposition~\ref{P:numberfields}$(ii)$ yields the following additional necessary conditions, for  every finite place $v\nmid m$ of $F$ with a unique
place $w$ of $K$ above it: If $a\simeq_F b$ then
\begin{itemize}
\item[(a)] $a_{i_0}$ and $b_{i_0}$ lie in the same coset of
$\langle\zeta^{\gcd(i_0,t_s)}\rangle N_{K_w/F_v}(K_w^\times)^s$ in $F_v^\times$; and
\item[(b)] after replacing $b$ by $\sigma_w^z(b)\nu$ with $\nu\in N_{K_w/F_v}(K_w^\times)^s$ so
that its $\beta^{i_0}$-coefficient equals $a_{i_0}$, the vectors
$(a_{i_1},\dots,a_{i_k})$ and $(b_{i_1},\dots,b_{i_k})$ lie in the same coset of
$\Delta_{t_s,I}$ in $(F_v^\times)^{[\times k]}$.
\end{itemize}
Here, $t_s$ and $\Delta_{t_s,I}$ are computed for $K_w/F_v$ by Lemmas~\ref{Lemma:tsandmunNorm} and
\ref{le:4.5}, and so depend on $v$.

\begin{remark}
 There are infinitely many places $v$ with a unique place $w$ of $K$. By
Proposition~\ref{P:numberfields} we now have a computable obstruction for  $a\simeq_F b$: if $a$ and $b$ are
represented by different elements of $\mathcal{S}_s(K_w)$ for one such $v\nmid m$, then
$K[t;\sigma]/K[t;\sigma](t^m-a)$ and $K[t;\sigma]/K[t;\sigma](t^m-b)$ are not
$F$-equivalent, and for $m=n$ they are not isomorphic.
\end{remark}


\providecommand{\bysame}{\leavevmode\hbox to3em{\hrulefill}\thinspace}
\providecommand{\MR}{\relax\ifhmode\unskip\space\fi MR }
\providecommand{\MRhref}[2]{%
  \href{http://www.ams.org/mathscinet-getitem?mr=#1}{#2}
}
\providecommand{\href}[2]{#2}

\end{document}